\documentclass{amsart}
\usepackage[english]{babel}
\usepackage{amsmath,amsfonts,amssymb,amsthm}
\usepackage[all]{xy}
\usepackage[mathcal]{eucal}
\usepackage{mathrsfs}
\usepackage{caption}
\usepackage{hyperref}
\newtheorem{lm}{Lemma}[section]
\newtheorem{prop}[lm]{Proposition}
\newtheorem{thm}[lm]{Theorem}
\newtheorem{cor}[lm]{Corollary}
\newtheorem{df}[lm]{Definition}

\theoremstyle{definition}
\newtheorem{rk}[lm]{Remark}

\newcommand{\C}{\mathbb{C}}
\newcommand{\R}{\mathbb{R}}
\newcommand{\val}{\mathrm{val}}
\newcommand{\Idlsh}{\mathscr{I}}
\newcommand{\regsh}{\mathcal{O}}
\newcommand{\dual}{{}^\vee}
\newcommand{\Q}{\mathbb{Q}}
\newcommand{\Z}{\mathbb{Z}}
\newcommand{\N}{\mathbb{N}}
\renewcommand{\div}{\mathrm{div}}
\newcommand{\Lscr}{\mathscr{L}}
\newcommand{\ord}{\mathrm{ord}}
\newcommand{\discrep}{\mathrm{discrep}}
\newcommand{\totaldiscrep}{\mathrm{totaldiscrep}}
\newcommand{\e}{\varepsilon}

\newcommand{\exc}{\mathrm{exc}}
\newcommand{\Rscr}{\mathscr{R}}
\newcommand{\Proj}{\mathrm{Proj}}
\newcommand{\Xtilde}{\widetilde{X}}
\newcommand{\dfrak}{\mathfrak{d}}
\newcommand{\Dfrak}{\mathfrak{D}}
\title{Log terminal singularities}
\author{alberto chiecchio and stefano urbinati}
\date{}
\begin{document}

%----------ABSTRACT-------------------------
\begin{abstract}
In this paper we give a new point of view for optimizing the definitions related to the study of singularities of normal varieties, introduced in \cite{dFH} and further studied in \cite{stefano} and \cite{stefano2}, in relation to the Minimal Model Program. We introduce a notion of discrepancy for normal varieties, and we define log terminal${}^+$ singularities. We use finite generation to relate these new singularities with log terminal singularities (in the sense of \cite{dFH}).
\end{abstract}
\maketitle
\tableofcontents

%----------Introduction----------
\section{Introduction}
In \cite{dFH} de Fernex and Hacon define a pullback for Weil divisors, giving two possible natural choices. In this paper we motivate, under the point of view of the Minimal Model Program, which one seems to be the most reasonable to work with and we further study properties of these singularities.

In \cite{dFH}, de Fernex and Hacon defined a notion of discrepancies for normal varieties. In particular, their approach is via introducing a pullback for Weil divisors, obtained with a limiting process (cf. Section 2). This new pullback has most of the expected properties, however it is not linear, i.e. when a divisor $D$ is not $\Q$-Cartier, in most cases $$f^*(-D) \neq -f^*(D).$$
This gives rise to two choices for the relative canonical divisor
$$K_{Y/X}^- := K_Y - f^*(K_X) \quad \mbox{and} \quad K_{Y/X}^+:= K_Y + f^*(-K_X),$$
where $f:Y \to X$ is any proper birational map of normal varieties. The first one is related to the notion of log terminal and log canonical singularity, whereas the second to the notion of canonical and terminal singularity. For this reason some pathologies may occur. Fox example, in \cite[4.1]{stefano} the author gives an example of a variety with singularities that are canonical but not log terminal. One of the main point of this work is to study the properties of singularities with respect to these two different relative canonical divisors, and relate them.

In sections 2 and 3 we review the main definitions of \cite{dFH}. We will prove that also $K^+_{Y/X}$ can be deterermined looking at (antieffective) boundaries, fact that \emph{a priori} had no reason to be (see corollary \ref{cor:K+ and boundaries}).

In section 4 we discuss limiting discrepancies, using the relative canonical divisor $K^-_{Y/X}$. We will show that these discepancies work as in the usual $\Q$-Gorenstein case. 

In section 5 we will focus on singularities determined by $K^+_{Y/X}$. In particular we will define \emph{log terminal${}^+$} singularities, which will correspond to $K^+_{Y/X}>-1$ (for every resolution $f:Y\rightarrow X$). One of the two main results of this work will be in this setting. We will prove the following.
\begin{thm}[Theorems \ref{thm:fgcr} and \ref{thm:kltiffanticanonicalfg}] If $X$ is a lt${}^+$ normal variety, then $\Rscr(X,K_X)$ is finitely generated over $X$. Moreover, in this case, $X$ is klt if and only if $\Rscr(X,-K_X)$ is finitely generated.
\end{thm}
This result is an extension in the lt${}^+$ setting of a result proved by the second author in the canonical case in \cite[3.7]{stefano2}. The finite generation of this ring is enough to construct a small birational morphism to the original variety that resembles the outcame of a flip. From the point of view of the MMP, thus, the finite generation of $\Rscr(X,-K_X)$ is not essential (remark \ref{rk:lt+fg}). Moreover, as klt singularities are lt${}^+$, this is a broader class of singularities where finite generation of $\Rscr(X,K_X)$ is known (remark \ref{rk:lt+fg}).

In section 6 will focus on the properties of divisors $D$ on a normal varieties $X$ such that $\Rscr(X,D)$ is finitely generated. To such a divisor we will associate an ideal sheaf $\Dfrak(D)\subseteq\regsh_X$. The construction of such ideal is based on the notion of defect ideals of \cite{BdFF}. Using this ideal we will prove the following theorem
\begin{thm}[Corollary \ref{cor:lt+ and D=O, then klt}] Let $X$ be a normal lt${}^+$ variety such that $\Dfrak(K_X)=\regsh_X$. Then $X$ has klt singularities.
\end{thm}
\subsection*{Acknowledgement}
A special thanks goes to our advisors, S\'andor Kov\'acs and Christopher Hacon, and to Tommaso de Fernex.

The authors would like to thank A. K\"uronya and M. Fulger for interesting conversations, and G. Pacienza for useful comments.

The authors started working jointly at this problem at the AIM workshop \emph{Relating test ideals and multiplier ideals}, organized by Karl Schwede and Kevin Tucker. Part of this work was done during a staying of the first author at University of Utah, and part during a staying of the second author at University of Washington.

%----------VALUATIONS AND DIVISORS----------
\section{Valuations and divisors}
Unless otherwise stated, all varieties are normal over $\C$ and all divisors are $\R$-Weil divisors, that is, $\R$-linear combinations of prime divisors. By \emph{log resolution} of a variety $X$, we mean a proper birational morphism $f:Y\rightarrow X$ from a smooth variety $Y$, with the exceptional locus being a simple normal crossing exceptional divisor.

Let $X$ be a normal variety. The basic idea of \cite{dFH} for pulling back Weil divisors is to think of them as sheaves in the function field of $X$ (which is a birational invariant). All the definitions in this section are contained in \cite{dFH}.

A \emph{divisorial valuation} on $X$ is a discrete valuation of the function field $k(X)$ of $X$ of the form $\nu=q\val_F$ where $q\in\R_{>0}$ and $F$ is a prime divisor over $X$, that is a prime divisor on some normal variety birational to $X$. Let $\nu$ be a discrete valuation. If $\Idlsh$ is a \emph{coherent fractional ideal} of $X$ (that is, a finitely generated sub-$\regsh_X$-module of the constant field $\mathcal{K}_X$ of rational functions on $X$), we set
$$
\nu(\Idlsh):=\min\{\nu(f)\,|\,\textrm{$f\in\Idlsh(U)$, $U\cap c_X(\nu)\neq\emptyset$}\}.
$$
If $Z$ is a proper closed subscheme of $X$, we set
$$
\nu(Z):=\nu(\Idlsh_Z),
$$
where $\Idlsh_Z$ is the ideal sheaf of $Z$. These definitions extend by linearity to $\R$-linear combinations of fractional ideals or closed subsets.

\begin{df} To a given fractional ideal $\Idlsh$ we can associate a divisor, called the \emph{divisorial part} of $\Idlsh$, as
$$
\div(\Idlsh):=\sum_{E\subset X}\val_E(\Idlsh)E,
$$
where the sum runs over all the prime divisors on $X$.
\end{df}
We notice that $\div(\Idlsh)$ can be characterized as the divisor on $X$ such that
$$
\regsh_X(-\div(\Idlsh))=\Idlsh\dual\dual.
$$

\begin{df} Let $\nu$ be a divisorial valuation on $X$. The \emph{$\natural$-valuation} or \emph{natural valuation} along $\nu$ of a divisor $F$ on $X$ is
$$
\nu^\natural(F):=\nu(\regsh_X(-F)).
$$
\end{df}
We remark that we could have chosen to evaluate the sheaf of $\regsh_X(F)$ instead, but the above definition is more natural in the philosophy of evaluating ideal sheaves for subschemes. Unfortunately, the natural valuation may fail to be additive even in the $\Q$-Cartier case (see \cite[2.3]{dFH}). As de Fernex and Hacon show (\cite[2.8]{dFH}), for every divisor $D$ on $X$ and every $m\in\Z_{>0}$, $m\nu^\natural(D)\geq\nu^\natural(mD)$ and
$$
\inf_{k\geq1}\frac{\nu^\natural(kD)}{k}=\liminf_{k\rightarrow\infty}\frac{\nu^\natural(kD)}{k}=\lim_{k\rightarrow\infty}\frac{\nu^\natural(k!D)}{k!}\in\R.
$$
\begin{df} Let $D$ be a divisor on $X$ and $\nu$ a divisorial valuation. The \emph{valuation along $\nu$} of $D$ is defined to be the above limit
$$
\nu(D):=\lim_{k\rightarrow\infty}\frac{\nu^\natural(k!D)}{k!}.
$$
\end{df}
We can use these valuations to pullback divisors.
\begin{df} Let $f:Y\rightarrow X$ be a proper birational morphism from a normal variety $Y$. For any divisor $D$ on $X$, the \emph{$\natural$-pullback} of $D$ along $f$ is
\begin{eqnarray}\label{eq:naturalpullbackdef}
f^\natural D:=\div(\regsh_X(-D)\cdot\regsh_Y).
\end{eqnarray}
\end{df}
As above, it can be characterized as the divisor on $Y$ such that
\begin{eqnarray}\label{eq:naturalpullbackchar}
\regsh_Y(-f^\natural D)=(\regsh_X(-D)\cdot\regsh_Y)\dual\dual.
\end{eqnarray}
\begin{df} We define the \emph{pullback} of $D$ along $f$ as
$$
f^*D:=\sum_{E\subset Y}\val_E(D)E,
$$
where the sum runs over all prime divisors on $Y$.
\end{df}
By definition, we have $f^*D=\liminf_k f^\natural(kD)/k$ coefficient-wise.

We notice that the evaluation along $\nu$ and the pullback defined above agrees with the corresponding notions in the case that the divisor $D$ is $\Q$-Cartier.
\begin{prop}\cite[2.4 and 2.10]{dFH} Let $\nu$ be a divisorial valuation on $X$ and let $f:Y\rightarrow X$ be a birational morphism from a normal variety $Y$. Let $D$ be any divisor let $C$ be any $\R$-Cartier divisor, with $t\in\R_{>0}$ such that $tC$ is Cartier.
\begin{enumerate}
\item The definitions of $\nu(C)$ and $f^*(C)$ given above coincide with the usual ones. More precisely,
$$
\nu(C)=\frac{1}{t}\nu(tC)\quad and \quad f^*(C)=\frac{1}{t}f^*(tC).
$$
Moreover,
$$
\nu(tC)=\nu^\natural(tC)\quad and\quad f^\natural(tC)=f^*(tC).
$$
\item The pullback is almost linear, in the sense that
\begin{eqnarray}\label{eq:pullbacklinear}
f^\natural(D+tC)=f^\natural(D)+f^*(tC)\quad and\quad f^*(D+C)=f^*(D)+f^*(C).
\end{eqnarray}
\end{enumerate}
\end{prop}
We observe that when working with natural the valuation and the natural pullback, the above properties are no longer true for $\R$-Cartier divisors which are not Cartier, see \cite[2.3]{dFH}. For example it may happen that $2C$ is Cartier, but $\nu^\natural(2C)\neq2\nu^\natural(C)$.
\begin{lm}\cite[2.7]{dFH}\label{lm:(fg)*-g*f*} Let $f:Y\rightarrow X$ and $g:V\rightarrow Y$ be two birational morphisms of normal varieties, and let $D$ be a divisor on $X$. The divisor $(f\circ g)^\natural(D)-g^\natural f^\natural(D)$ is effective and $g$-exceptional. Moreover, if $\regsh_X(-D)\cdot\regsh_Y$ is an invertible sheaf, $(f \circ g)^\natural(D)=g^\natural f^\natural(D)$.
\end{lm}
\begin{cor}\cite[2.13]{dFH}\label{cor:(fg)*-g*f*} Let $f:Y\rightarrow X$ and $g:V\rightarrow Y$ be two birational morphisms of normal varieties, and let $D$ be a divisor on $X$. The divisor $(f\circ g)^*(D)-g^*f^*(D)$ is effective and $g$-exceptional. Moreover, if $\regsh_X(-mD)\cdot\regsh_Y$ is an invertible sheaf for all sufficiently divisible $m$, $(f\circ g)^*(D)=g^*f^*(D)$.
\end{cor}
\begin{proof} The result comes from the fact that
$$
f^*(D)=\liminf\frac{1}{m}f^\natural(mD)
$$
coefficient-wise.
\end{proof}

%----------BOUNDARIES-----------------------
\section{Relative canonical divisors and boundaries}
If $f:Y\rightarrow X$ is a proper birational morphism (of normal varieties) and if we choose a canonical divisor $K_X$ on $X$, we will always assume that the canonical divisor $K_Y$ on $Y$ be chosen such that $f_*K_Y=K_X$ (as Weil divisors).

As in \cite{dFH}, we introduce the following definitions (with a slightly different notation\footnote{De Fernex and Hacon denote by $K_{m,Y/X}$ what we denote by $K^-_{m,Y/X}$ and by $K_{Y/X}$ what we denote by $K^+_{Y/X}$. Our notation is simply meant to help us keep track of the different relative canonical divisors})
\begin{df} Let $f:Y\rightarrow X$ be a proper birational map of normal varieties. The \emph{$m$-limiting relative canonical $\Q$-divisors} are
\begin{eqnarray*}
&&K^-_{m,Y/X}:=K_Y-\frac{1}{m}f^\natural(mK_X),\quad K^-_{Y/X}:=K_Y-f^*(K_X)\\
&&K^+_{m,Y/X}:=K_Y+\frac{1}{m}f^\natural(-mK_X),\quad K^+_{Y/X}:=K_Y+f^*(-K_X).
\end{eqnarray*}
\end{df}
As shown by \cite{dFH}, for all $m,q\geq1$,
\begin{eqnarray}\label{eq:K-<=K^+}
K^-_{m,Y/X}\leq K^-_{qm,Y/X}\leq K^-_{Y/X}\leq K^+_{Y/X}\leq K^+_{mq,Y/X}\leq K^+_{m,Y/X}
\end{eqnarray}
and
\begin{eqnarray}\label{K-(Y/X)=limsupK-(m,Y/X)}
K^-_{Y/X}=\limsup K^-_{m,Y/X}, \quad K^+_{Y/X}=\liminf K^+_{m,Y/X}
\end{eqnarray}
(coefficient-wise). A priori, we could use these two notions to define the various flavors of singularities, and there is no reason to suspect that one behaves better than the other. The idea is to study these singularities in the context of the MMP.
\begin{rk} If $X$ is a normal vairety, by \emph{boundary} on $X$ we will mean a $\Q$-divisor $\Delta$ on $X$ such that $K_X+\Delta$ is $\Q$-Cartier. We do not make any assumption on the effectiveness of $\Delta$.
\end{rk}
\begin{df} Let $\Delta$ be a boundary on $X$ and let $f:Y\rightarrow X$ be a proper birational morphism. The \emph{log relative canonical $\Q$-divisor} of $(Y,f_*^{-1}\Delta)$ over $(X,\Delta)$ is given by
$$
K^\Delta_{Y/X}:=K_Y+f_*^{-1}\Delta-f^*(K_X+\Delta)=K_Y+f_*^{-1}\Delta+f^*(-K_X-\Delta)
$$
where $f^{-1}_*\Delta$ is the strict transform of $\Delta$ on $Y$.
\end{df}

\begin{rk}\label{rk:dFH 3.9} With the same computation as \cite[3.9]{dFH}, we find that, if $\Delta$ is a boundary for $X$ and $m\geq1$ is such that $m(K_X+\Delta)$ is Cartier,
$$
\begin{array}{ll}
K^-_{m,Y/X}=K^\Delta_{Y/X}-\frac{1}{m}f^\natural(-m\Delta)-f_*^{-1}\Delta,\hspace{0.7em}&K^-_{Y/X}=K^\Delta_{Y/X}-f^*(-\Delta)-f_*^{-1}\Delta,\\
K^+_{m,Y/X}=K^\Delta_{Y/X}+\frac{1}{m}f^\natural(m\Delta)-f_*^{-1}\Delta,&K^+_{Y/X}=K^\Delta_{Y/X}+f^*(\Delta)-f_*^{-1}\Delta.
\end{array}
$$
Notice that, if $m$ is such that $m(K_X+\Delta)$ is Cartier and $\Delta$ is \emph{effective}, we have
$$
K^\Delta_{Y/X}\leq K^-_{m,Y/X}\leq K^-_{Y/X}\leq K^+_{Y/X}\leq K^+_{m,Y/X}.
$$
On the other hand, if $\Delta$ is \emph{anti-effective} and $m(K_X+\Delta)$ is Cartier,
$$
K^\Delta_{Y/X}\leq K^+_{m,Y/X}.
$$ 
\emph{A priori}, there is no clear relation between $K^\Delta_{Y/X}$ and $K^+_{Y/X}$.
\end{rk}
\begin{lm}\label{lm:K+ and boundaries} Let $X$ be a normal variety, and let $f:Y\rightarrow X$ be a proper birational morphism of normal varieties. Let $\Delta$ be an \emph{effective} divisor such that $K_X-\Delta$ is $\Q$-Cartier. Then
$$
K^+_{Y/X}\leq K^{-\Delta}_{Y/X}.
$$
\end{lm}
\begin{proof} We have to compare
$$
K^+_{Y/X}=K_Y+f^*(-K_X)
$$
and
$$
K^{-\Delta}_{Y/X}=K_Y+f_*^{-1}(-\Delta)-f^*(K_X-\Delta).
$$
We have
\begin{eqnarray*}
K^+_{Y/X}&=&K_Y+f^*(-K_X)=K_Y+f^*(-K_X+\Delta-\Delta)=\\
&=&K_Y+f^*(-K_X+\Delta)+f^*(-\Delta)=\\
&=&K_Y-f^*(K_X-\Delta)+f^*(-\Delta).
\end{eqnarray*}
In order to obtain the desired inequality, we need to show that
$$
f^*(-\Delta)\leq f^{-1}_*(-\Delta).
$$
We claim that, if $-D$ is anti-effective $\Q$-divisor such that $mD$ is integral,
$$
f^*(-D)\leq\frac{f^\natural(-mD)}{m}\leq f^{-1}_*(-D).
$$
The first inequality is by definition of pullback. To prove the second inequality, we can assume directly that $D$ is an integral divisor, and show that $f^\natural(-D)\leq f^{-1}_*(-D)$. Notice that, if $D$ is a Weil divisor on a normal variety, $D\geq0$ if and only if $\regsh_X(-D)\subseteq\regsh_X$. By duality, $-D\leq0$ if and only if $\regsh_X(D)\supseteq\regsh_X$. Then, since $-D\leq0$,
$$
\regsh_Y\subseteq\regsh_X(D)\cdot\regsh_Y\subseteq(\regsh_X(D)\cdot\regsh_Y)\dual\dual=\regsh_X(-f^\natural(-D)).
$$
Therefore, $f^\natural(-D)=f^{-1}_*(-D)+F\leq0$, where $F$ is the exceptional component of $f^\natural(-D)$. Therefore, it must be $F\leq0$, and thus $f^\natural(-D)\leq f^{-1}_*(-D)$.
\end{proof}
One of the main technical results that enable us to study singularities in this setting is due to \cite{dFH}. We will reproduce the proof for the convenience of the reader.
\begin{lm}\label{lm:compatible D-boundary} Let $X$ be a normal variety, and let $f:Y\rightarrow X$ be a proper birational map, with $Y$ normal. Let $D$ be a Weil divisor on $X$. For each $m\geq2$ there exists a divisor $\Delta_m$ such that
\begin{enumerate}
\item $\Delta_m\geq0$;
\item $D+\Delta_m$ is $\Q$-Cartier;
\item $\lfloor\Delta_m\rfloor=0$ and $m\Delta_m$ is integral; and
\item $\frac{f^\natural(mD)}{m}=f^*(D+\Delta_m)-f^{-1}_*\Delta_m$.
\end{enumerate}
Moreover, if $f$ is a sufficiently high resolution of $X$, $\Delta_m$ can be chosen such that $\mathrm{exc}(f)\cup f^{-1}_*\Delta_m$ has simple normal crossing support. We will call such divisor an \emph{$m$-compatible $D$-boundary with respect to $f$}.
\end{lm}
\begin{rk}\label{rk:boundary} We recall that, by \emph{log resolution} of $X$ we mean a resolution with exceptional locus $\mathrm{exc}(f)$ a simple normal crossing divisor. If $D=K_X$, we will simply talk of \emph{compatible $m$-boundary}.

This definition is slightly different from the one in \cite[5.1]{dFH}. Indeed, if $\Delta_m$ is an $m$-compatible boundary, we are not asking that the map $f$ be a resolution for the pair $(X,\Delta_m)$. For that reason, the log discrepancy of the pair $(X,\Delta_m)$ with respect to $f$ is not necessarily the log discrepancy of the pair $(X,\Delta_m)$. In particular, it may be that the pair $(X,\Delta_m)$ has worse singularities that what is suggested by the map $f$ alone.
\end{rk}
\begin{proof} The proof will follow \cite[5.4]{dFH}. Let $T$ be an effective divisor such that $D-T$ is Cartier. Let $\Lscr$ be an line bundle such that $\Lscr\otimes\regsh_X(-mT)$ is generated by global sections, and let $G$ be a general element in the linear system $\{H\in|\Lscr|\,, H-mT\geq0\}$. Let $M=G-mT$ and
$$
\Delta_m=\frac{M}{m}.
$$
We notice right away that $\Delta_m\geq0$ and $\lfloor\Delta_m\rfloor=0$. Moreover, $m(D+\Delta_m)=mD+M=mD-mT+G=m(D-T)+G$ is Cartier.  Let $E$ be an irreducible exceptional divisor of $f$. Since $G$ is general enough, $\ord_E G=\ord_E(mT)$. Thus, $\ord_E(-M)=\ord_E(-G+mT)=-\ord_E(G)+\ord_E(mT)=0$, and therefore $f^\natural(-M)=f^{-1}_*(-M)$. Then,
\begin{eqnarray*}
f^\natural(mD)&=&f^\natural(m(D+\Delta_m)-m\Delta_m)=mf^*(D+\Delta_m)+f^\natural(-m\Delta_m)=\\
&=&mf^*(D+\Delta_m)+f^\natural(-M)=mf^*(D+\Delta_m)-f^{-1}_*(M)=\\
&=&mf^*(D+\Delta_m)-mf^{-1}_*(\Delta_m).
\end{eqnarray*}

If $f$ is a resolution of $\regsh_X(-mD)+\regsh_X(-mT)$, then $\Delta_m$ can be chosen general enough so that $\exc(f)\cup f^{-1}_*\Delta_m$ has simple normal crossing support.
\end{proof}.
\begin{rk} We immediately have that $K^-_{Y/X}=\sup K^\Delta_Y/X$, where $\Delta$ run among the compatible $K_X$-boundaries. 
\end{rk}
Lemma \ref{lm:K+ and boundaries} gives the following (non-immediate) fact.
\begin{cor}\label{cor:K+ and boundaries} If $f:Y\rightarrow X$ is a proper birational morphism of normal varieties,
$$
K^+_{Y/X}=\inf K^{-\Delta}_{Y/X},
$$
where $\Delta$ run among the compatible $(-K_X)$-boundaries.
\end{cor}

%----------DISCREPANCIES I------------------
\section{Discrepancies and singularities}
\begin{df} Let $Y\rightarrow X$ be a proper birational morphism with $Y$ normal, and let $F$ be a prime divisor on $Y$. For each integer $m\geq1$, the \emph{$m$-limiting discrepancy} of $F$ with respect to $X$ is
$$
a_m(F,X):=\ord_F(K^-_{m,Y/X}).
$$
The \emph{discrepancy} of $F$ with respect to $X$ is
$$
a(F,X):=\ord_F(K^-_{Y/X}).
$$
\end{df}
We recall that, from \cite{komo}, if $\Delta$ is a boundary for $K_X$, we have
$$
a(F,X,\Delta):=\ord_F(K^\Delta_{Y/X}).
$$
Lemma \ref{lm:compatible D-boundary} and \eqref{K-(Y/X)=limsupK-(m,Y/X)} can be rephrased as follows (see also \cite[2.12]{stefano}).
\begin{cor}\label{cor:supofdiscrepancies} Let $f:Y\rightarrow X$ be a proper birational morphism between normal varieties, and let $F$ be a divisor on $Y$. For any $m\geq2$ there is an $m$-compatible boundary such that
$$
a_m(F,X)=a(F,X,\Delta_m).
$$
In particular
\begin{eqnarray}
\label{eq:a=sup}a(F,X)&=&\sup\{a_m(F,X)\}=\\
&=&\sup\{a(F,X,\Delta)\,|\,\textrm{$(X,\Delta)$ is a log pair, $\Delta\geq0$}\}\nonumber
\end{eqnarray}
\end{cor}

Finally, we can introduce the definitions of singularities.  We can use limiting discrepancies to define singularities.
\begin{df}\label{df:sing} The variety $X$ is said to satsfy condition $\mathrm{M}_{\geq-1}$ (resp. $\mathrm{M}_{>-1}$, resp. $\mathrm{M}_{\geq0}$, resp. $\mathrm{M}_{>0}$) if there is an integer $m_0$ such that $a_m(F,X)\geq-1$ (resp. $>-1$, resp. $\geq0$, resp. $>0$) for every prime divisor $F$ over $X$ and $m=m_0$ (and hence for any positive multiple $m$ of $m_0$).
\end{df}
We recall that conditions $\mathrm{M}_{\geq-1}$ and $\mathrm{M}_{>-1}$ are used by \cite{dFH} to define singularites.
\begin{df}[\cite{dFH}, 7.1] The variety $X$ is called \emph{log canonical} (resp. \emph{log terminal}) if it satisfies condition $\mathrm{M}_{\geq-1}$ (resp., $\mathrm{M}_{>-1}$).
\end{df}
Some results that justify the introduction of these conditions are the following.
\begin{thm}\label{thm:dFH 7.2} A variety $X$ satisfies condition $\mathrm{M}_{\geq-1}$ (resp. $\mathrm{M}_{>-1}$, resp. $\mathrm{M}_{\geq0}$, resp. $\mathrm{M}_{>0}$) if and only if there is an effective boundary $\Delta$ such that $(X,\Delta)$ is log canonical (resp. Kawamata log terminal, resp. canonical, resp. terminal).
\end{thm}
\begin{proof} For the $\mathrm{M}_{\geq-1}$ and $\mathrm{M}_{>-1}$ cases, this is \cite[7.2]{dFH}. The same proof given there proves the other cases. We repeat the proof for the convenience of the reader.

Let us first assume that there exists a boundary $\Delta$ on $X$ giving the singularity type. Let $m$ be an integer, such that $m(K_X+\Delta)$ is Cartier. By remark \ref{rk:dFH 3.9}, for any proper birational morphism $f:Y\rightarrow X$, with $Y$ normal, and any divisor $F$ on $Y$, $K^\Delta_{Y/X}\leq K^-_{m,Y/X}$. Hence $X$ will have singularities no worse than those given by the pair.

Conversely,  let $X$ satisfy condition $\mathrm{M}_{\geq-1}$ (resp. $\mathrm{M}_{>-1}$, resp. $\mathrm{M}_{\geq0}$, resp. $\mathrm{M}_{>0}$) and let $m_0$ be chosen as in definition \ref{df:sing}. Let $\Delta_{m_0}$ be an $m_0$-compatible boundary for a log resolution $f:Y_{m_0}\rightarrow X$ of $((X,\Delta_{m_0}),\regsh_X(-m_0K_X))$ (see lemma \ref{lm:compatible D-boundary}). For each exceptional divisor $F$ on $Y_{m_0}$, $a_{m_0}(F,X)=a(F,X,\Delta_{m_0})$, which proves that $(X,\Delta_{m_0})$ is log canonical (resp. log terminal, resp. canonical, resp. terminal).
\end{proof}
\begin{lm} Let $X$ be a normal $\Q$-Gorenstein variety. Then $X$ satisfies condition $\mathrm{M}_{\geq-1}$ (resp. $\mathrm{M}_{>-1}$, resp. $\mathrm{M}_{\geq0}$, resp. $\mathrm{M}_{>0}$) if and only if $X$ is log canonical (resp. Kawamata log terminal, resp. canonical, resp. terminal).
\end{lm}
\begin{lm} Let $X$ be a normal variety. If it satisfies $\mathrm{M}_{>0}$, then it satisfies $\mathrm{M}_{\geq0}$. If it satisfies $\mathrm{M}_{\geq0}$, then it satisfies $\mathrm{M}_{>-1}$. If it satisfies $\mathrm{M}_{>-1}$, then it satisfies $\mathrm{M}_{\geq-1}$.
\end{lm}
Essentially, we have the chain of implications
$$
\mathrm{M}_{>0}\Rightarrow\mathrm{M}_{\geq0}\Rightarrow\mathrm{M}_{>-1}\Rightarrow\mathrm{M}_{\geq-1}.
$$
We conclude this section by showing that the notion of discrepancy can work for the non-$\Q$-Gorenstein case exactly as it works in the $\Q$-Gorenstein case.
\begin{df} Let $X$ be a normal variety. The \emph{discrepancy} of $X$ is given by
$$
\discrep(X)=\inf_E\{\textrm{$a(E,X)$, $E$ is an exceptional divisor over $X$}\}
$$
(where $E$ runs through all the irreducible exceptional divisors of all (proper) birational morphisms $Y\rightarrow X$).
\end{df}
The standard theory of discrepancies can be extended to this setting, as done by \cite{dFH}. One result that is different in this case is the following (see \cite[2.30]{komo}).
\begin{lm}\label{lm:komo 2.30} Let $X$ be a normal variety, $f:Y\rightarrow X$ be a resolution of singularities and let $\Delta_Y$ be the $\Q$-divisor on $Y$ such that
$$
K_Y+\Delta_Y = f^*(K_X),\quad f_*\Delta_Y=0.
$$
For any prime divisor $F$ over $X$,
$$
a(F,X)\leq a(F,Y,\Delta_Y),
$$
with equality if 
\begin{enumerate}
\item $F$ is a divisor on $Y$,
\item $K_X$ is $\Q$-Cartier, or
\item $\regsh_X(-m(K_X))\cdot\regsh_Y$ is invertible for all sufficiently divisible $m$.
\end{enumerate}
\end{lm}
The above lemma can be formulated more generally in the sense of \cite{dFH}, for pairs $(X,Z)$ and assuming that $f:Y\rightarrow X$ is just a proper birational morphism of normal varieties. However, this is beyond the purpose of this paper.
\begin{proof} The proof proceeds as in the log $\Q$-Gorenstein case. Let $g:V\rightarrow Y$ be a resolution with $F$ a divisor on $V$, and let $h=f \circ g$. We have the diagram
$$
\xymatrix{ V \ar[d]_g \ar[dr]^h & \\
Y \ar[r]_f & X.}
$$
By corollary \ref{cor:(fg)*-g*f*},
$$
(f \circ g)^*(K_X)=g^*f^*(K_X)+B=g^*(K_Y+\Delta_Y)+B,
$$
where $B$ is effective and $g$-exceptional. Therefore, if $A'$ is the discrepancy of $h$ with respect to $X$ and $A$ is the discrepancy of $g$ with respect to $(Y,\Delta_Y)$, then
$$
A'=K_V-(f \circ g)^*(K_X)=K_V-g^*(K_Y+\Delta_Y)-B=A-B,
$$
that is $A=A'+B$, with $B$ effective and $g$-exceptional. Since the same computation holds for any resolution $g:V\rightarrow Y$, this proves the inequality in general. Since $B$ is $g$-exceptional, (a) is immediate. Case (b) is a particular case of (c), and (c) follows again from corollary \ref{cor:(fg)*-g*f*} ($K_Y+\Delta_Y$ is $\Q$-Cartier, since $Y$ is smooth).
\end{proof}
\begin{rk} In the $\Q$-Gorenstein case, we have the following (\cite[2.30]{komo}). If $\Delta_X$ is a boundary on $X$, $f:Y\rightarrow X$ is a proper birational morphism of normal varieties, and $\Delta_Y$ is a divisor on $Y$ such that $K_Y+\Delta_Y\equiv f^*(K_X+\Delta_X)$ and $f_*\Delta_Y=\Delta_X$, then $\discrep(X,\Delta_X)=\totaldiscrep(Y,\Delta_Y)$, and this is the standard way of computing discrepancies. In our case, if $(Y,\Delta_Y)\rightarrow X$ are as in lemma \ref{lm:komo 2.30}, we only have
$$
\discrep(X)\leq\totaldiscrep(Y,\Delta_Y).
$$
\end{rk}
As in the $\Q$-Gorenstein case, we have the following result (see \cite[2.31(a)]{komo}).
\begin{cor}\label{cor:-1<=discrep<=1} Let $X$ be a normal variety. Then $\discrep(X)=-\infty$ or $-1\leq\discrep(X)\leq1$.
\end{cor}
\begin{proof}Blowing up a point of codimension $2$ contained in the smooth locus we have $\discrep(X)\leq1$. If $\discrep(X)<-1$, let $E$ be on a resolution $f:Y\rightarrow X$ and such that $a(E,X)<-1$. Let $\Delta_Y$ be as in lemma \ref{lm:komo 2.30}. Then $E$ appears in $\Delta_Y$ with coefficient $-a(E,X)$, and $\discrep(X)\leq\totaldiscrep(Y,\Delta_Y)=-\infty$.
\end{proof}
This implies that the case $\discrep(X)=-1$ is the broadest class where it makes sense to define discrepancies. We have the following immediate consequences.
\begin{lm} Let $X$ be a normal $\Q$-Gorenstein variety. Then $\discrep(X)\geq-1$  (resp. $>-1$, resp. $\geq0$, resp. $>0$) if and only if $X$ is log canonical (resp. Kawamata log terminal, resp. canonical, resp. terminal).
\end{lm}
\begin{rk} Since $K^-_{m,Y/X}\leq K^-_{Y/X}$, see remark \ref{rk:dFH 3.9}, if $X$ is a normal variety satisfying $\mathrm{M}_{\geq-1}$ (log canonical), then $\discrep(X)\geq-1$. Similarly if $X$ satisfies $\mathrm{M}_{>-1}$ (log terminal), then $\discrep(X)>-1$; if it satisfies $\mathrm{M}_{\geq0}$, then $\discrep(X)\geq0$; if it satisfies $\mathrm{M}_{>0}$, then $\discrep(X)>0$. 
\end{rk}

%----------DISCREPANCIES--------------------
\section{Lt${}^+$ singularities}
\subsection{Definition, and a criterion}
In the previous section, we used the relative canocial divisors $K^-_{m,Y/X}$ and $K^-_{Y/X}$ to define singularities. Another possibility is to use $K^+_{Y/X}$. Following \cite[8.1]{dFH}, but with a different notation we introduce the following.
\begin{df} Let $Y\rightarrow X$ be a proper birational morphism with $Y$ normal, and let $F$ be a prime divisor on $Y$. The \emph{discrepancy${}^+$} of $F$ with respect to $X$ is
$$
a^+(F,X):=\ord_F(K^+_{Y/X}).
$$
\end{df}
We recall that in \cite{dFH}, a normal variety $X$ is defined \emph{canonical} (resp. \emph{terminal}), if $a^+(F,X)\geq0$ (resp. $a^+(F,X)>0$) for all prime divisors $F$, exceptional over $X$.
\begin{df} Let $X$ be a normal variety. we say that $X$ has \emph{log terminal${}^+$}, or simply, \emph{lt${}^+$}, singularities if $a^+(F,X)>-1$ for all prime divisors $F$, exceptional over $X$.
\end{df}
\begin{lm}\label{lm:klt=>lt+} Let $X$ be a normal variety. If $X$ has (Kawamata) log terminal singularities, then it is lt${}^+$. The converse is true if $X$ is $\Q$-Gorenstein. 
\end{lm}
For a generic normal varity $X$, asking that $a^+(F,X)>-1$ for all $F$, or asking that $\inf\{a^+(F,X)\}>-1$ is the same, as the next lemma shows.
\begin{lm}\label{lm:ltwithoneresolution} Let $X$ be a normal variety, and let $f:Y\rightarrow X$ be a log resolution, i.e. the exceptional locus of $f$ is a simple normal crossing divisor. If $a^+(E,X)>-1+\e$ for all prime exceptional divisors $E$ on $Y$, for some $\e>0$, then $a^+(F,X)>-1+\e$ for all prime divisors $F$ exceptional over $X$. In particular, then, $X$ is lt${}^+$.
\end{lm}
\begin{rk}\label{rk:afterltwithoneresolution} Notice that the hypothesis of this lemma is satisfied if $a^+(E,X)>-1$ for all prime exceptional divisors $E$ on $Y$, since there are only finitely many such divisors.
\end{rk}
\begin{proof} Let $h:Z\rightarrow X$ be any other resolution, that we can assume is a log resolution factoring through $f$, and let $g:Z\rightarrow Y$ so that $h=f\circ g$. Let $F_Y$ be the reduced exceptional divisor of $f$, and $F_Z$ be the one of $h$. By \cite[2.13]{dFH},
$$
h^*(-K_X)-g^*f^*(-K_X)=E^{+,g}\geq0,
$$
which is effective and $g$-exceptional. Then
\begin{eqnarray*}
K^+_{Z/X}+(1-\e)F_Z&=&K_Z+h^*(-K_X)+(1-\e)F_Z=\\
&=&K_Z+g^*f^*(-K_X)+(1-\e)F_Z+E^{+,g}=\\
&=&K_Z-g^*(K_Y+(1-\e)F_Y)+(1-\e)F_Z+\\
&&+g^*\big(K_Y+f^*(-K_X)+(1-\e)F_Y\big)+\\
&&+E^{+,g}=\\
&=&K_Z-g^*(K_Y+(1-\e)F_Y)+(1-\e)F_Z+\\
&&+g^*\big(K^+_{Y/X}+(1-\e)F_Y\big)+\\
&&+E^{+,g}
\end{eqnarray*}
We only need to check this expression on the prime divisors on $Z$ which are exceptional over $Y$. Then, $K_Z-g^*(K_Y+(1-\e)F_Y)+(1-\e)F_Z$ has positive valuations, since $(Y,(1-\e)F_Y)$ has log terminal singularites; the term $K^+_{Y/X}+(1-\e)F_Y$ is effective by assumption, and thus so is its pullback; $E^{+,g}$ is effective, as noticed above.
\end{proof}

In the same spirit as \cite[8.2]{dFH}, we have the following result.
\begin{prop}\label{prop:sheaf criterion} Let $X$ be a normal variety. Then $X$ is lt${}^+$ if and only if there exists $\e\in\Q$, $\e>0$, such that, for all sufficiently divisible $m\geq1$, and for all resolutions of $X$,
$$
\regsh_X(mK_X)\cdot\regsh_Y\subseteq\regsh_Y\big(m(K_Y+(1-\e)F_Y)\big),
$$
where $F_Y$ is the reduced exceptional divisor of $f$.
\end{prop}
\begin{proof} Let $X$ be lt${}^+$. By lemma \ref{lm:ltwithoneresolution}, there exists $\e$ such that $a^+(E,X)>-1+\e$ for all prime divisors $E$, exceptional over $X$ (choose an arbitrary log resolution, and choose the value of $\e$ that works for that resolution: it will work for all). We can choose $\e\in\Q$. If $f$ and $m$ are chosen as in the statement, with $m$ divisible enough so that $m\e\in\N$,
\begin{eqnarray*}
m(K_Y+(1-\e)F_Y)+f^\natural(-mK_X)&\geq&m(K_Y+(1-\e)F_Y)+f^*(-mK_X)=\\
&=&m\big(K_Y+(1-\e)F_Y+f^*(-K_X)\big)\geq\\
&\geq&0.
\end{eqnarray*}
Thus,
\begin{eqnarray*}
\regsh_X(mK_X)\cdot\regsh_Y&\subseteq&\big(\regsh_X(mK_X)\cdot\regsh_Y\big)\dual\dual=\\
&=&\regsh_Y(-f^\natural(-mK_X))\subseteq\regsh_Y\big(m(K_Y+(1-\e)F_Y)\big).
\end{eqnarray*}
Conversely, suppose that $X$ is not lt${}^+$. Let $E$ by a prime divisor, exceptional over $X$, such that $a^+(E,X)\leq-1$. Let $f:Y\rightarrow X$ be a resolution, with $E$ as one of the components of the exceptional divisor. Since $a^+(E,X)\leq-1$, for all $\e\in\Q$, $\e>0$, there exists $m$ such that $\val_E(K_Y+\frac{1}{m}f^\natural(-mK_X))<-1+\e$. The same will be true for all $m'\geq m$, so we can assume that $m\e\in\N$. By further blowing up, we can assume that $f$ is a resolution of $\regsh_X(mK_X)$, since the value of $\val_E(K_Y+\frac{1}{m}f^\natural(-mK_X))$ does not actually depend on the smooth model where we compute it. But then, 
$$
\regsh_X(mK_X)\cdot\regsh_Y=\big(\regsh_X(mK_X)\cdot\regsh_Y\big)\dual\dual\nsubseteq\regsh_Y\big(m(K_Y+(1-\e)F_Y)\big).
$$
\end{proof}
\begin{rk}\label{rk:after sheaf criterion} The necessary condition is still true if we substitute $\e$ with any other $0<\e'\leq\e$, or if $Y\rightarrow X$ is just a proper birational map between normal varieties. Moreover, the divisibility condition on $m$ is only to ensure that $m\e\in\Z$.
\end{rk}
\subsection{The canonical ring}
Let us first recall some standard definitions from the Minimal Model Program (for references, \cite{komo}, \cite{kol}, \cite{bchm}, and \cite{hk}).
\begin{df} Let $X$ be a normal projective variety and let $D$ be a $\Q$-divisor on $X$. We define
$$
R(X,D):=\bigoplus_{m\geq1}H^0(\regsh_X(\lfloor mD\rfloor))
$$
and
$$
\Rscr(X,D):=\bigoplus_{m\geq1}\regsh_X(\lfloor mD\rfloor).
$$
If $f:X\rightarrow U$ is a projective morpishm of quasi-projective normal varieties we define the relative version
$$
\Rscr(f,D):=\bigoplus_{m\geq1}f_*\regsh_X(\lfloor mD\rfloor).
$$
\end{df}
\begin{rk} All the above have a natural ring structure. Moreover, $\Rscr(X,D)$ is an $\regsh_X$-algebra, while $\Rscr(f,D)$ is an $\regsh_U$-algebra.
\end{rk}
With the above criterion, proposition \ref{prop:sheaf criterion}, we can prove the following result (which was originally proven in the canonical case by the second author in \cite[3.6]{stefano2}).
\begin{thm}\label{thm:fgcr} If $X$ is a lt${}^+$ normal variety, then $\Rscr(X,K_X)$ is finitely generated.
\end{thm}
\begin{proof} We may assume that $X$ is affine.
Let $p:\Xtilde \rightarrow X$ be a resolution (which we can assume projective, \cite[0.2]{komo}). By \cite{bchm} $\Rscr{}(p,K_{\tilde X})$ is finitely generated. Running the relative Minimal Model Program for $\Xtilde$ over $X$, we obtain the model $X^c=\Proj_X\Rscr(p,K_{\tilde X})$, with induced morphism $f:X^c\rightarrow X$. Let $F_{X^c}$ be the reduced exceptional divisor of $f$. The variety $X^c$ has Kawamata log terminal singularities.

Since $X$ is lt${}^+$, there is a rational $\e>0$ such that, for any $m>0$ sufficiently divisible, there is an inclusion
$$
\mathcal O _{X^c}\cdot\regsh _X(mK_X)\subseteq\regsh _{X^c}\big(m(K_{X^c}+(1-\e)F_{X^c})\big).
$$
To be precise, if $d$ is a positive integer such that $d\e\in\Z$, then for any $m\geq1$,
$$
\mathcal O _{X^c}\cdot\regsh _X(dmK_X)\subseteq\regsh _{X^c}\big(dm(K_{X^c}+(1-\e)F_{X^c})\big)
$$
(see proposition \ref{prop:sheaf criterion} and remark \ref{rk:after sheaf criterion}). Pushing this forward we obtain inclusions
$$
f_*(\regsh_{X^c}\cdot\regsh_X(dmK_X))\subseteq f_*\regsh_{X^c}\big(dm(K_{X^c}+(1-\e)F_{X^c})\big)\subseteq\regsh_X(dmK_X).
$$
Since the left and right hand sides have isomorphic global sections, then
$$
H^0\big(f_*\regsh_{X^c}\big(m(dK_{X^c}+d(1-\e)F_{X^c})\big)\big)\cong H^0(\regsh_X(mdK_X)).
$$
Since $X$ is affine, $\regsh_X(mK_X)$ is globally generated and hence
$$
f_*\regsh_{X^c}\big(m(dK_{X^c}+d(1-\e)F_{X^c})\big)=\regsh_X(mK_X).
$$
But then $\Rscr(X,dK_X)\cong\Rscr(f,dK_{X^c}+d(1-\e)F_{X^c})$ is finitely generated over $\regsh_X$, since $X^c$ has Kawamata log terminal singularities (\cite[92]{kol}). Finally, since $\Rscr(X,dK_X)=\Rscr(X,K_X)^{(d)}$ is finitely generated, then $\Rscr(X,K_X)$ is finitely generated over $\regsh_X$ (\cite[5.68]{hk}).
\end{proof}

\begin{rk}\label{rk:allsamealgs} Note that we have seen that 
$$
\Rscr(X,dK_X)\cong\Rscr(f,dK_{X^c}+d(1-\e)F_{X^c})\cong\Rscr(p,dK_{\Xtilde}).
$$
Hence
\begin{eqnarray*}
X^c&=&\Proj_X\Rscr(p,K_{\Xtilde})\cong\Proj_X\Rscr(p,K_{\Xtilde})^{(d)}=\Proj_X\Rscr(p,dK_{\Xtilde})\cong\\
&\cong&\Proj_X\Rscr(X,dK_X)=\Proj_X\Rscr(X,K_X)^{(d)}\cong\Proj_X\Rscr(X,K_X),
\end{eqnarray*}
and with this chain of isomorphisms $\regsh_{X^c}(1)$ is preserved (\cite[Ex II.5.13]{har}), and so $X^c \to X$ is a small mophism. \emph{A posteriori} then, we don't have an exceptional divisor for $f$, and we have inclusions
$$
\mathcal O _{X^c}\cdot\regsh _X(mK_X)\subseteq\regsh _{X^c}(mK_{X^c})
$$
for all $m\geq1$. Moreover, we have isomorphisms
$$
\Rscr(X,K_X)\cong\Rscr(f,K_{X^c})\cong\Rscr(p,K_{\Xtilde}).
$$
\end{rk}
\begin{rk}\label{rk:lt+fg} Notice that if $(X,\Delta)$ is a klt pair, then $X$ has lt${}^+$ singularities (lemma \ref{lm:klt=>lt+}), while there are examples of singularities which are lt${}^+$, but cannot be klt (\cite[4.1]{stefano}). Thus, lt${}^+$ is a class broader than klt for which we have finite generation of the algebra $\Rscr(X,K_X)$. Of course, our results relies heavily on \cite{bchm}.
\end{rk}
\begin{cor}\label{cor:canmod}  If $X$ is lt${}^+$, then the relative canonical model $X_{can}=\Proj_X\Rscr(X,K_X)$ exists and has log terminal singularities.
\end{cor}
\begin{proof} From the proof of theorem \ref{thm:fgcr} and remark \ref{rk:allsamealgs}, we have the existence of the relative canonical model $X_{can}=\Proj_X\Rscr(X,K_X)=X^c$, with log terminal singularities.
\end{proof}
\begin{rk} Notice that, using the methods of \cite[Proposition 4.4]{stefano}, lemma \ref{lm:ltwithoneresolution}, and remark \ref{rk:afterltwithoneresolution}, we can show directly that $\Proj_X\Rscr(X,K_X)$ has log terminal singularities (assuming the finite generation of $\Rscr(X,K_X)$). We reproduce the argument here, since is similar to the one we will use to prove theorem \ref{thm:kltiffanticanonicalfg}.

Since $X$ is lt${}^+$, for any sufficiently high log resolution $f:Y\rightarrow X$, we have $K_Y-\frac{1}{m}f^\natural(-mKX)>-1$. By \cite[lemma 6.2]{komo}, there exists a small birational morphism $\pi:X^+\rightarrow X$ such that $K_{X^+}$ is a $\pi$-ample $\Q$-Cartier divisor. Also, for this morphism we have that $\pi^\natural(-mK_X)=-mK_{X^+}$. Let us consider a common log resolution of $X$ and $X^+$, $f:Y\rightarrow X$ and $g:Y\rightarrow X^+$.

Let us consider the map $\regsh_{X^+}\cdot\regsh_X(mK_X)\rightarrow\regsh_{X^{+}}(mK_{X^+})$. Since $\pi^{-1}_*(K_X)=K_{X^+}$ is $\pi$-ample, $\regsh_{X^+}(mK_{X^+})$ is globally generated over $X$ for $m$ sufficiently divisible, hence we have an isomorphism of sheaves. Thus
$$
K_Y-g^*(K_{X^+})=K_Y+\frac{1}{m}g^*(-mK_{X^+})=K_Y+\frac{1}{m}f^\natural(-mK_X)>-1,
$$
where the last equality holds by [dFH09, Lemma 2.7]. By lemma \ref{lm:ltwithoneresolution} and remark \ref{rk:afterltwithoneresolution}, the canonical model $X^+$ has lt${}^+$ singularities, that is, log terminal singularities since it is $\Q$-Gorenstein.
\end{rk}
\subsection{Finite generation and singularities}
Here we prove a result, which is again a generalization of a result of the second author, in \cite[3.7]{stefano2}.
\begin{thm}\label{thm:kltiffanticanonicalfg} Let $X$ be a normal variety with lt${}^+$ singularities. Then, $\Rscr(X,-K_X)$ is a finitely generated $\regsh_X$-algebra if and only if $X$ is log terminal.
\end{thm}
\begin{proof} If $X$ is log terminal, then $\Rscr(X,-K_X)$ is finitely generated by \cite[Theorem 92]{kol}.

Let $\Rscr(X,-K_X)$ be finitely generated. By \cite[Lemma 6.2]{komo} there is a small map $\pi:X^-\rightarrow X$ such that $\pi^{-1}_*(-K_X)$ is $\Q$-Cartier and  $\pi$-ample. For any $m$ sufficiently divisible, consider the natural map $\regsh_{X^-}\cdot\regsh_X(-mK_X)\rightarrow\regsh_{X^-}(-mK_{X^-})$, which, since the map is small, is an isomorphism of sheaves. Thus, considering $f:Y\rightarrow X$ and $g:Y\rightarrow X^-$, a common log resolution, we have
$$
K_Y+\frac{1}{m}g^*(-mK_{X^-})=K_Y+\frac{1}{m}g^*(\pi^\natural(-mK_X))=K_Y+\frac{1}{m}f^\natural(-mK_X)>-1.
$$
This means that $X^-$ has at most lt${}^+$ singularities. Since $K_{X^-}$ is $\Q$-Cartier, $X^-$ is log terminal.

Choosing a general ample $\Q$-divisor $H^-\sim_{\Q,X}-K_{X^-}$, let $1||m$, and $G^-\in|mH^-|$ be a general irreducible divisor. Then, setting $\Delta^-:=\frac{G^-}{m}$, we have that $K_{X^-}+\Delta^-\sim_{\Q,X}0$ is still log terminal and $\pi^*(K_X+\pi_*\Delta^-)=_\Q K_{X^-}+\Delta^-$. Hence $(X,\Delta=\pi_*\Delta^-)$ is log terminal.
\end{proof}
\begin{cor} Let $X$ be a lt${}^+$ normal variety. Then $X$ is log terminal if and only if $\Rscr(X,D)$ is finitely generated for any Weil divisor $D$.
\end{cor}
\begin{proof} If $X$ is log terminal, then this is \cite[Theorem 92]{kol}. Conversely, in particular $\Rscr(X,-K_X)$ is finitely generated, and by the previous theorem, $X$ is log terminal.
\end{proof}
\begin{rk}\label{rk:flip} In the spirit of corollary \ref{cor:canmod}, if $X$ has klt singularities, the birational map
$$
\xymatrix{\Proj_X\Rscr(X,-K_X) \ar@{-->}[rr] \ar[dr] & & \Proj_X\Rscr(X,K_X) \ar[dl]\\
& X &}
$$
is a flip. In particular, if searching for minimal (or canonical) models, the finite generation of the algebra $\Rscr(X,-K_X)$ is not essential.
\end{rk}

%----------DEFECT IDEALS-----------------
\section{Defect ideals}

In this section we focus on the behavior of the Defect Ideals introduced by Boucksom, de Fernex and Favre in \cite{BdFF} in the context of log terminal singularities according to \cite{dFH}. 

\begin{df} Let $X$ be a normal variety and $D$ any divisor on $X$ and $m$ any positive natural number. Let $f_m:Y_m \to X$ be a log resolution of $(X, \regsh_X(mD)\otimes \regsh_X(-mD))$. The $m$-th defect ideal of $D\subseteq X$ is given by
$$
\dfrak^m(D):= f_{m,*} \regsh_Y \left(\lfloor \frac{1}{m}(-f_m^{\natural}(-mD)-f_m^{\natural}(mD)) \rfloor \right).
$$
\end{df}

Notice that $-f_m^{\natural}(-mD)-f_m^{\natural}(mD)\leq0$, so that indeed $\dfrak^m(D)$ is an ideal.

It is in general very hard to control the behavior of the defect ideals when $m$ changes. We will prove that, when the divisorial ring associated to $D$ is finitely generated (with no assumptions on the ring associated to $-D$), then the sequence of ideals stabilizes when $m$ is big enough. 

\begin{thm}
Let $X$ be a normal variety and $D$ a Weil divisor such that $\Rscr(X, D)$ is finitely generated. Then the limit of the defect ideals exists and it is an ideal:
$$
\Dfrak(D):= \limsup_{m\to \infty} f_{m,*}\regsh_{Y_m} \left(\lfloor\frac{1}{m}(-f_m^{\natural}(-mD)-f_m^{\natural}(mD))\rfloor \right).
$$
\end{thm}
\begin{proof}
For every map $f: Y \to X$, we already know that $f^{\natural}(D)\geq \frac{f^{\natural}(mD)}{m}$. Even more, if $\Rscr(X, D)$ is finitely generated, there exists an $m_0$ such that $f^{\natural}(-mkD) = kf^{\natural}(-mD)$ for every $k>0$. Hence we have (choosing a resolution $f_{mk}$ high enough) 
\begin{eqnarray*}
\dfrak^{mk}(D)&=& f_{km,*} \regsh_Y \left(\lfloor \frac{1}{mk}(-f_{mk}^{\natural}(-mkD)-f_{mk}^{\natural}(mkD)) \rfloor\right)=\\
&=& f_{km,*} \regsh_Y \left(\lfloor \frac{-f_{mk}^{\natural}(-mD)}{m}- \frac{f_{mk}^{\natural}(mkD)}{mk} \rfloor\right) \supseteq\\
&\supseteq & f_{km,*} \regsh_Y \left(\lfloor \frac{1}{m}(-f_{mk}^{\natural}(-mD)-f_{mk}^{\natural}(mD)) \rfloor\right)=\\
&=& \mathfrak{d}^m(D).
\end{eqnarray*}

By noetherianity these ideals will stabilize for $k \gg 0$.
\end{proof}

\begin{cor} If $\Dfrak(D)=\regsh_X $ then $D$ is numerically Cartier in the sense of \cite{BdFF}, and $\Rscr(-D,X)$ is finitely generated.
\end{cor}
\begin{proof} Let $\Dfrak(D)=\dfrak^m(D)=\dfrak^{mk}(D)=\regsh_X$, for some $m$ and all $k>0$. Since $\dfrak^{km}(D)=\regsh_X$, then, for any resolution $f$ high enough
$$
\lfloor-\frac{1}{mk}(-f^\natural(-mkD)-f(mkD))\rfloor=0.
$$
Since the quantity inside the round down is non-positive, it must be
$$
f^\natural(-mkD)=-f^\natural(mkD).
$$
More precisely, we have the above identity for a given resolution $f_{mk}$ and for any other higher resolution. But since this is an equality of Weil divisors, by pushforward we have the same identity on any model.

The first consequence of this is that, for any resolution $f$ of $X$, $f^*(-D)=-f^*(D)$, that is, $D$ is numerically Cartier.

The second consequence is the finite generation of $\Rscr(X,-D)$, and this comes from observing the relation between finite generation and numerically Cartier (see, for eaxample, \cite{dFH}, 2.2). In this case, we have that, since $D$ is numerically Cartier, $\regsh_X(mD)\cdot\regsh_X(-mD)=\regsh_X$ for $m$ sufficiently divisible (\cite{dFH}, Definitions 2.22 and 2.24). But then, for any $q>0$ and $m$ sufficiently divisible,
\begin{eqnarray*}
\regsh_X(-mD)^q&\cong&\regsh_X(-mD)^q\otimes\regsh_X(mqD)\otimes\regsh_X(-mqD)\cong\\
&\cong&\regsh_X(-mD)^q\otimes\regsh_X(mD)^q\otimes\regsh_X(-mqD)\cong\\
&\cong&\regsh_X(-mqD).
\end{eqnarray*}
The first isomorphism comes from what observed earlier, while the second one from the finite generation of $\Rscr(X,D)$. Thus $\Rscr(X,-mD)=\Rscr(X,-D)^{(m)}$ is finitely generated, which implies that $\Rscr(X,-D)$ is finitely generated.
\end{proof}
\begin{rk} The above proof shows that, if $D$ is such that $\Rscr(X,D)$ is finitely generated and $D$ is numerically Cartier, then $\Rscr(X,-D)$ is finitely generated.
\end{rk}
\begin{cor}\label{cor:lt+ and D=O, then klt} Let $X$ be a normal variety with lt${}^+$ singularities and such that $\Dfrak(K_X)=\regsh_X$ (or such that $X$ is numerically Cartier). Then $X$ is klt.
\end{cor}

%----------BIBLIOGRAPHY---------------------

\end{document}